\documentclass[12pt]{amsart}

\usepackage{amssymb}

\newtheorem{theorem}{Theorem}[section]
\newtheorem{proposition}[theorem]{Proposition}

\newtheorem{corollary}[theorem]{Corollary}

\begin{document}
\baselineskip=15pt

\title[Symplectic form on hyperpolygon spaces]{Symplectic form
on hyperpolygon spaces}

\author[I. Biswas]{Indranil Biswas}

\address{School of Mathematics, Tata Institute of Fundamental
Research, Homi Bhabha Road, Bombay 400005, India}

\email{indranil@math.tifr.res.in}

\author[C. Florentino]{Carlos Florentino}

\address{Departamento Matem\'atica, Centro de An\'alise Matem\'atica,
Geometria e Sistemas din\^amicos -- LARSYS, Instituto Superior
T\'ecnico, Av. Rovisco Pais, 1049-001 Lisbon, Portugal}
\email{cfloren@math.ist.utl.pt}

\author[L. Godinho]{Leonor Godinho}

\address{Departamento Matem\'atica, Centro de An\'alise Matem\'atica,
Geometria e Sistemas din\^amicos -- LARSYS, Instituto Superior
T\'ecnico, Av. Rovisco Pais, 1049-001 Lisbon, Portugal}

\email{lgodin@math.ist.utl.pt}

\author[A. Mandini]{Alessia Mandini}

\address{Departamento Matem\'atica, Centro de An\'alise Matem\'atica,
Geometria e Sistemas din\^amicos -- LARSYS, Instituto Superior
T\'ecnico, Av. Rovisco Pais, 1049-001 Lisbon, Portugal}

\email{amandini@math.ist.utl.pt}

\subjclass[2000]{14D20, 14H60}

\keywords{Hyperpolygon space, parabolic Higgs bundles, symplectic form, 
Liouville form}

\date{}

\begin{abstract}
In \cite{GM}, a family of parabolic Higgs bundles on ${\mathbb C}{\mathbb P}^1$
has been constructed and identified with a moduli space of hyperpolygons. Our aim here is
to give a canonical  alternative construction of this family. This enables us
to compute the Higgs symplectic form for this family and show that the isomorphism
of \cite{GM} is a symplectomorphism.
\end{abstract}

\maketitle

\section{Introduction}
Hyperpolygon spaces, $X(\alpha)$, with 
$\alpha \in (\mathbb{R}_{\geq 0})^n$, were introduced by Konno in \cite{K} as the
hyper-K\"{a}hler
analogue of polygon spaces. They are defined as the hyper-K\"{a}hler quotients of the cotangent bundle
$T^* \mathbb{C}^{2n}$ by the group
\begin{equation}\label{eq:group}
K \,:= \,\Big({\rm U}(2) \times {\rm U}
(1)^n\Big)/{\rm U}(1)\,=\, \Big({\rm SU}(2) \times 
{\rm U}(1)^n\Big)/ (\mathbb{Z}/2\mathbb{Z})\, ;
\end{equation}
in the definition of $K$,
the nontrivial element of $\mathbb{Z}/2\mathbb{Z}$ acts diagonally as
multiplication by $-1$ on each factor (see \cite{K} for details).
It is shown in \cite{GM} there exists an isomorphism between $X(\alpha)$
and a family of parabolic Higgs bundles over ${\mathbb C}{\mathbb P}^1$
that depends on $\alpha$. More precisely, $X(\alpha)$ is isomorphic
to the family of parabolic Higgs bundles $(E,\Phi)$ on $(\mathbb{C}
\mathbb{P}^1\, , \{x_i\}_{i=1}^n)$, where $x_i$ are fixed marked points,
$E$ is a holomorphically trivial vector bundle over $\mathbb{C} 
\mathbb{P}^1$ of rank two with a weighted complete flag structure over each of
the marked points 
\begin{align*}
E_{x_i,1}\,\supsetneq\, E_{x_i,2} \,\supsetneq\, 0\, , \\
0 \,\leq\, \beta_1(x_i) \,< \,\beta_2(x_i) \,<\,1\, ,
\end{align*}
with $\beta_2(x_i)-\beta_1(x_i)\,=\,\alpha_i$. 

Here we give an alternative canonical construction of these families of parabolic Higgs
bundles. This construction enables us to prove that the pullback, by the above
isomorphism, of the Higgs symplectic form for each family to the corresponding hyperpolygon
space $X(\alpha)$ coincides with the symplectic form on $X(\alpha)$ obtained by reduction from the Liouville symplectic form on
$T^* \mathbb{C}^{2n}$. As a consequence, the isomorphism constructed in \cite{GM} is
a symplectomorphism.   

\section{A family of parabolic Higgs bundles on
${\mathbb C}{\mathbb P}^1$}\label{sec2}

Fix $n$ distinct ordered points $D\,:=\,\{x_1\, ,\cdots \, , x_n\} \, \subset\,
{\mathbb C}{\mathbb P}^1$, with $n\, \geq\, 3$. For each
$i\, \in\, [1\, ,n]$,
fix a real number $\alpha_i\,\in\, (0\, ,1)$.
Let ${\mathcal M}^H_P$ be the moduli stack of parabolic Higgs
bundles over ${\mathbb C}{\mathbb P}^1$ of rank two and degree
zero with parabolic divisor $D$ and parabolic weights
$\{\alpha_i\, ,0\}$ at $x_i$, for  $i\, \in\, [1\, ,n]$.

Let
$$
U\, :=\, ({\mathbb C}^2\setminus \{0\})^{n}
\, \subset\, ({\mathbb C}^2)^{\oplus n}
$$
be the Zariski open subset. The algebraic cotangent bundle
$T^*U$ is the trivial vector bundle over $U$ with fiber
$(({\mathbb C}^2)^{\oplus n})^*$. So the total space of $T^*U$
is identified with the Cartesian product $(({\mathbb C}^2)^{\oplus n})^*\times U$.
For each $i\, \in\, [1\, ,n]$, let
\begin{equation}\label{e0}
\overline{f}_i\, : \,U\, \longrightarrow\, {\mathbb C}^2
\setminus \{0\}~\,~\text{and}~\,~
\overline{g}_i\, : \,(({\mathbb C}^2)^{\oplus n})^*\,
\longrightarrow\, ({\mathbb C}^2)^*
\end{equation}
be the projections to the $i$--th factor of the Cartesian products.
For $v\, \in\, {\mathbb C}^2$ and $w\, \in\,
({\mathbb C}^2)^*$, we have $v\otimes w\, \in\,
{\mathbb C}^2\otimes ({\mathbb C}^2)^*\, =\,
{\rm End}_{\mathbb C}({\mathbb C}^2)$, and $w(v)\, \in\,
{\mathbb C}$. Define the Zariski closed subscheme of
$(({\mathbb C}^2)^{\oplus n})^*\times U$
\begin{equation}\label{e-1}
{\mathcal Z}\, :=\, \left\{ (y,z) \, \in\, (({\mathbb C}^2)^{\oplus n})^*\times U\,
\mid\  \sum_{i=1}^n \overline{f}_i(z)\otimes
\overline{g}_i(y)\, =\, 0~\,\text{~and~}\,~
\overline{g}_i(y)\left(\overline{f}_i(z)\right)\, =\, 0\, \, \forall
\,\, i\right\}\, .
\end{equation}
(Note that $\overline{f}_i(z)\otimes\overline{g}_i(y)\,\in\,
{\rm End}_{\mathbb C}({\mathbb C}^2)$.)

The section of the trivial rank two vector bundle over ${\mathbb C}^2\setminus \{0\}$
$${\mathbb C}^2\times ({\mathbb C}^2\setminus \{0\})\,\longrightarrow
\, {\mathbb C}^2\setminus \{0\}$$ defined by $v\, \longmapsto\,
(v\, ,v)$ will be denoted by $s_0$. Let
\begin{equation}\label{e1}
L_0\, \subset \, {\mathbb C}^2\times ({\mathbb C}^2\setminus \{0\})
\,\longrightarrow\, {\mathbb C}^2\setminus \{0\}
\end{equation}
be the line subbundle generated by the section $s_0$.
Let $f_i$ be the composition
\begin{equation}\label{e2}
{\mathcal Z}\, \longrightarrow\, U \,
\stackrel{\overline{f}_i}{\longrightarrow}\,
{\mathbb C}^2\setminus \{0\}\, ,
\end{equation}
where $\overline{f}_i$ and $\mathcal Z$ are constructed in \eqref{e0}
and \eqref{e-1} respectively, and the map
${\mathcal Z}\, \longrightarrow\, U$ is the natural projection.

We will construct a morphism from $\mathcal Z$ to the moduli stack
${\mathcal M}^H_P$. This amounts to constructing a parabolic
Higgs vector bundle over ${\mathbb C}{\mathbb P}^1\times
\mathcal Z$ of the given type.

The vector bundle underlying the parabolic bundle will be the
trivial vector bundle of rank two
\begin{equation}\label{cv}
{\mathcal V}\, :=\,
{\mathbb C}^2\times({\mathbb C}{\mathbb P}^1\times {\mathcal Z})
\, \longrightarrow\, {\mathbb C}{\mathbb P}^1\times {\mathcal Z}
\end{equation}
over ${\mathbb C}{\mathbb P}^1\times {\mathcal Z}$. For
each point $x_i\, \in\, D$, we have the line subbundle over ${\mathcal Z}$
\begin{equation}\label{e3}
{\mathcal L}_i\, :=\, f^*_i L_0\, \subset\, 
{\mathcal V}\vert_{\{x_i\}\times {\mathcal Z}}\, =\,
{\mathbb C}^2\times {\mathcal Z}\, ,
\end{equation}
where $f_i$ and $L_0$ are constructed in \eqref{e2} and
\eqref{e1} respectively. The quasiparabolic filtration over
$\{x_i\}\times {\mathcal Z}$ is given by the line subbundle
${\mathcal L}_i$. The parabolic weight of ${\mathcal L}_i$
is $\alpha_i$ and the parabolic weight of 
${\mathcal V}\vert_{\{x_i\}\times {\mathcal Z}}$ is $0$.

Let us denote the holomorphic cotangent bundle of ${\mathbb C}{\mathbb P}^1$
by $K_{{\mathbb C}{\mathbb P}^1}$. Consider the short exact sequence
of coherent sheaves
$$
0\, \longrightarrow\, K_{{\mathbb C}{\mathbb P}^1}
\, \longrightarrow\,K_{{\mathbb C}{\mathbb P}^1}\otimes{\mathcal
O}_{{\mathbb C}{\mathbb P}^1}(D) \, \longrightarrow\,
\left(K_{{\mathbb C}{\mathbb P}^1}\otimes{\mathcal
O}_{{\mathbb C}{\mathbb P}^1}(D)\right)\vert_D\, \cong \,
{\mathcal O}_D \, \longrightarrow\, 0
$$
over ${\mathbb C}{\mathbb P}^1$; the above identification
of $(K_{{\mathbb C}{\mathbb P}^1}\otimes{\mathcal
O}_{{\mathbb C}{\mathbb P}^1}(D))\vert_D$ with ${\mathcal O}_D$
is given by the Poincar\'e adjunction formula \cite[p. 146]{GH}. Tensoring this exact
sequence with ${\rm End}_{\mathbb C}({\mathbb C}^2)$, and then
taking the corresponding long exact sequence of cohomologies, we get
\begin{equation}\label{ef1}
0\, \longrightarrow\, H^0({\mathbb C}{\mathbb P}^1,\,
{\rm End}_{\mathbb C}({\mathbb C}^2)\otimes
K_{{\mathbb C}{\mathbb P}^1}\otimes{\mathcal
O}_{{\mathbb C}{\mathbb P}^1}(D))\, \longrightarrow\,
{\rm End}_{\mathbb C}({\mathbb C}^2)\otimes H^0(D,\, {\mathcal O}_D)
\end{equation}
$$
\, =\, {\rm End}_{\mathbb C}({\mathbb C}^2)^{\oplus n}
\, \stackrel{\phi}{\longrightarrow}\, H^1({\mathbb C}{\mathbb P}^1,
\, {\rm End}_{\mathbb C}({\mathbb C}^2)\otimes
K_{{\mathbb C}{\mathbb P}^1})
$$
because $H^0({\mathbb C}{\mathbb P}^1, K_{{\mathbb C}{\mathbb P}^1})
\, =\, 0$. Using Serre duality, we have
$$H^1({\mathbb C}{\mathbb P}^1,\, K_{{\mathbb C}{\mathbb P}^1})\,=\,
H^0({\mathbb C}{\mathbb P}^1,\, {\mathcal O}_{{\mathbb C}{\mathbb P}^1})^*
\,=\,\mathbb C\, .$$ Therefore,
$$
H^1({\mathbb C}{\mathbb P}^1,\,
{\rm End}_{\mathbb C}({\mathbb C}^2)\otimes
K_{{\mathbb C}{\mathbb P}^1})\,=\,
{\rm End}_{\mathbb C}({\mathbb C}^2)\otimes
H^1({\mathbb C}{\mathbb P}^1,\, K_{{\mathbb C}{\mathbb P}^1})\,=\,
{\rm End}_{\mathbb C}({\mathbb C}^2)\, .
$$
The homomorphism
$$
\phi\, :\, {\rm End}_{\mathbb C}({\mathbb C}^2)^{\oplus n}\,
\longrightarrow\, {\rm End}_{\mathbb C}({\mathbb C}^2)
$$
in the exact sequence \eqref{ef1} coincides with the one defined
by 
$$(A_1\, ,\cdots \, ,A_n)\, \longmapsto\, \sum_{i=1}^n A_i.$$
Now we will generalize these over a family. 

Let $p\,:\, {\mathbb C}{\mathbb P}^1\times
{\mathcal Z} \, \longrightarrow\, {\mathbb C}{\mathbb P}^1$
be the natural projection. Let $\mathcal{A}:={\mathbb C}[{\mathcal Z}]$
be the $\mathbb C$-algebra defined by the algebraic functions on the
scheme ${\mathcal Z}$. For notational convenience,
the vector bundle $\text{End}({\mathcal V})\,\longrightarrow\,
{\mathbb C}{\mathbb P}^1$, where $\mathcal V$ is
defined in \eqref{cv}, will be denoted by
$\widetilde{\mathcal V}$. Consider the short exact sequence
of coherent sheaves
\begin{equation}\label{e7}
0\, \longrightarrow\, \widetilde{\mathcal V}\otimes
p^*K_{{\mathbb C}{\mathbb P}^1}
\, \longrightarrow\, \widetilde{\mathcal V}\otimes p^*(K_{{\mathbb C}
{\mathbb P}^1}\otimes{\mathcal
O}_{{\mathbb C}{\mathbb P}^1}(D)) \, \longrightarrow\,
\widetilde{\mathcal V}\vert_{D\times \mathcal Z} \, \longrightarrow\, 0
\end{equation}
over ${\mathbb C}{\mathbb P}^1\times \mathcal Z$, where the line bundle
$\left(p^*(K_{{\mathbb C} {\mathbb P}^1}\otimes{\mathcal
O}_{{\mathbb C}{\mathbb P}^1}(D))\right)\vert_{D\times \mathcal Z}$ is identified
with ${\mathcal O}_{D\times \mathcal Z}$ using the Poincar\'e
adjunction formula (as done before). We have
$$
H^0({\mathbb C}{\mathbb P}^1\times{\mathcal Z},
\, \widetilde{\mathcal V}\otimes
p^*K_{{\mathbb C}{\mathbb P}^1})\, =\, 0 ~\,\text{~and~}\,~
H^1({\mathbb C}{\mathbb P}^1\times {\mathcal Z},
\, \widetilde{\mathcal V}\otimes
p^*K_{{\mathbb C}{\mathbb P}^1})\, =\, \text{End}({\mathbb C}^2)
\otimes_{\mathbb C}{\mathcal A}
$$
because $H^1({\mathcal Z},\, {\mathcal O}_{\mathcal Z}) \, =\, 0$
(recall that $\mathcal Z$ is a Zariski open subset of an affine
variety). Also,
$$
H^0({\mathbb C}{\mathbb P}^1\times {\mathcal Z},\, 
\widetilde{\mathcal V}\vert_{D\times \mathcal Z})\, =\,
\text{End}({\mathbb C}^2)\otimes(\oplus_{i=1}^n {\mathcal A})\, =\,
(\text{End}({\mathbb C}^2)\otimes_{\mathbb C}{\mathcal A})^{\oplus n}\, .
$$
Therefore, the long exact sequence of cohomologies associated to the
exact sequence in \eqref{e7} gives
$$
0\, \longrightarrow\, H^0\left({\mathbb C}{\mathbb P}^1\times {\mathcal Z},
\, \widetilde{\mathcal V}\otimes p^*(K_{{\mathbb C}
{\mathbb P}^1}\otimes{\mathcal
O}_{{\mathbb C}{\mathbb P}^1}(D))\right)\, \longrightarrow\,
(\text{End}({\mathbb C}^2)\otimes{\mathcal A})^{\oplus n}
\, \stackrel{\overline\phi}{\longrightarrow}\,
\text{End}({\mathbb C}^2)\otimes{\mathcal A}\, ,
$$
where the above homomorphism $\overline\phi$ sends any
$(A_1\, ,\cdots \, ,A_n) \, \in\, (\text{End}({\mathbb C}^2)\otimes
{\mathcal A})^{\oplus n}$ to $\sum_{i=1}^n A_i$.

Consequently, from the condition $\sum_{i=1}^n \overline{f}_i(z)\otimes
\overline{g}_i(y)\, =\, 0$ in \eqref{e-1} it follows that
there is a unique algebraic section
$$
\theta_0 \, \in\, H^0\left({\mathbb C}{\mathbb P}^1\times {\mathcal Z},\,
\text{End}({\mathcal V})\otimes p^*(K_{{\mathbb C}{\mathbb P}^1}\otimes
{\mathcal O}_{{\mathbb C}{\mathbb P}^1}(D))\right)
$$
such that for any $(y,z)\,\in\, {\mathcal Z}$ and $x_i\,\in\, D$,
$$
\theta_0(x_i\, ,y,z)\, =\, \overline{f}_i(z)\otimes
\overline{g}_i(y)\, \in\,{\rm End}_{\mathbb C}({\mathbb C}^2)\, .
$$
{}From the condition $\overline{g}_i(y)\left(\overline{f}_i(z)\right)\, =\, 0$
in \eqref{e-1} it follows immediately that $\theta_0(x_i\,,y ,z)$
is nilpotent with respect to the quasiparabolic filtration
(constructed in \eqref{e3}). Therefore, $\theta_0$ defines a Higgs
field for the family of parabolic vector bundles. In other words,
$$
({\mathcal V}\, ,\{{\mathcal L}_i\}_{i=1}^n\, ,\theta_0)
$$
is a family of parabolic Higgs bundles on  
${\mathbb C}{\mathbb P}^1$ parametrized by $\mathcal Z$.
Hence we get a morphism
\begin{equation}\label{eq:morph}
\varphi_0\, :\, {\mathcal Z}\, \longrightarrow\,{\mathcal M}^H_P\, .
\end{equation}

Let ${\mathcal M}^{H,S}_P\, \subset\, {\mathcal M}^H_P$ be the
moduli stack of stable parabolic Higgs bundles.
We choose $\{\alpha_i\}_{i=1}^n$ in such a
way that the image of $\varphi_0$ constructed
in \eqref{eq:morph} intersects the stable locus ${\mathcal M}^{H,S}_P$.
Then from the openness of the stability condition, \cite{Ma},
it follows that there is a nonempty Zariski open subset
\begin{equation}\label{cus}
{\mathcal U}_S\, \subset\, \mathcal Z
\end{equation}
such that image of $\varphi_0\vert_{{\mathcal U}_S}$ is in ${\mathcal M}^{H,S}_P$,
in other words,
\begin{equation}\label{e5}
\varphi_0\vert_{{\mathcal U}_S}\, :\,
{\mathcal U}_S\, \longrightarrow\,
{\mathcal M}^{H,S}_P\, \subset\, {\mathcal M}^H_P\, .
\end{equation}

Let $M^H_P$ be the moduli space of stable parabolic Higgs bundles
over ${\mathbb C}{\mathbb P}^1$ of rank two and degree
zero with parabolic divisor $D$ and parabolic weights
$\{\alpha_i\, ,0\}$ at $x_i$, $i\, \in\, [1\, ,n]$. Consider
the morphism ${\mathcal M}^{H,S}_P\, \longrightarrow\, M^H_P$
to the coarse moduli space. Let
\begin{equation}\label{e6}
\varphi\, :\, {\mathcal U}_S\, \longrightarrow\, M^H_P
\end{equation}
be its composition with the morphism in \eqref{e5}.

Let $G$ be the complexification $K_{\mathbb{C}}\,=\, \Big({\rm SL}(2,\mathbb{C}) \times 
(\mathbb{C}^*)^n\Big)/ (\mathbb{Z}/2\mathbb{Z})\, ,$ of $K$ defined in \eqref{eq:group}. Then $G$ acts on  $(({\mathbb C}^2)^{\oplus n})^*\times U$ by
\begin{align}\label{eq:action}
[A,\lambda_1,\cdots, \lambda_n]\cdot ((y_1,\cdots,y_n), & (z_1,\cdots,z_n))
\\ \nonumber & =  ((\lambda_1^{-1} y_1 A, \cdots, \lambda_n^{-1} y_n A),(A^{-1 } z_1 \lambda_1, \cdots,A^{-1} z_n\lambda_n)),
\end{align}
keeping the space $\mathcal{Z}$ invariant.

There is a notion of $\alpha$-stability for hyperpolygons coming from hyper-K\"ahler quiver varieties \cite{N,K}. In \cite{GM} it is shown 
 that  $X(\alpha)$, obtained as a  GIT quotient using this $\alpha$-stability, coincides with the GIT quotient, $\mathcal{U}_S/ \!\! / G$, of  $\mathcal{U}_S$ by $G$. Indeed, the space $\mathcal{Z}$ is precisely the $0$-level set of the complex moment map for the hyper-K\"ahler action of $G$ on $(({\mathbb C}^2)^{\oplus n})^*\times U$ as considered in \cite{K,GM} and the open dense subset $\mathcal{U}_S$ coincides with the set of $\alpha$-stable elements in this level set (see  \cite[Theorem 3.1]{GM}). This also implies that  the morphism $\varphi$ is $G$-invariant.
 
  Let $\mathcal{H}(\alpha)\subset M^H_P$ be the image of this morphism. Note that 
$\mathcal{H}(\alpha)$ is the space of parabolic Higgs bundles in $M^H_P$ whose underlying vector bundle is holomorphically trivial. We  have  shown the following result.

\begin{proposition}\label{prop:GM}
The morphism $\varphi$ in \eqref{e6} is $G$-invariant, inducing the isomorphism 
$$
\overline{\varphi}\,:\,\mathcal{U}_S/ \!\! / G=X(\alpha)\,\longrightarrow
\, \mathcal{H}(\alpha)
$$
constructed in \cite[Theorem 3.1]{GM}.
\end{proposition}

\section{Symplectic structure on the moduli of parabolic
Higgs bundles}

\subsection{A natural $1$--form}

The moduli space $M^H_P$ has a natural algebraic $1$--form \cite{Hi},
\cite{BR}; we will recall its construction.

Take a stable parabolic Higgs bundle
\begin{equation}\label{epb}
{\bf E}\,:=\,(E\, ,\{\ell_i\}_{i=1}^n\, , \theta) \,\in\, M^H_P\, 
\end{equation}
over $\mathbb{C}\mathbb{P}^1$, where $\ell_i$ is a line in $E_{x_i}$
giving the quasiparabolic filtration over $x_i\,\in\, D$. Let
$$
\text{End}_P(E)\, \subset\, \text{End}(E) \, :=\, E\otimes E^*
$$
be the subsheaf given by locally defined endomorphisms
$s$ such that $s(x_i)(\ell_i) \, \subset\, \ell_i$ for all
$x_i$ in the domain of definition of $s$. Let
$$
\text{End}^0_P(E)\, \subset\, \text{End}_P(E)
$$
be the subsheaf given by the locally defined endomorphisms $s$ such that 
$s(x_i)(\ell_i) \,=\, 0$ and $s(x_i)(E_{x_i}) \, \subset\,\ell_i$ for all $x_i$ 
in the domain of definition of $s$. So both $\text{End}_P(E)\vert_{\mathbb{C}\mathbb{P}^1\setminus D}$ and 
$\text{End}^0_P(E)\vert_{\mathbb{C}\mathbb{P}^1\setminus D}$ are identified with $\text{End}(E)\vert_{\mathbb{C}\mathbb{P}^1\setminus D}$. For holomorphic sections $s$ and $t$ of $\text{End}_P(E)$ and $\text{End}^0_P(E)
\otimes {\mathcal O}_{{\mathbb C}{\mathbb P}^1}(D)$ respectively, both
defined over an open subset $U_0$ of ${\mathbb C}{\mathbb P}^1$, the 
composition $s\circ t$ is a holomorphic section of $\text{End}(E)\otimes
{\mathcal O}_{{\mathbb C}{\mathbb P}^1}(D)$ over $U_0$ which is nilpotent over
the points of $D\cap U_0$;
here ``$\circ$'' is the composition of endomorphisms of $E$. Therefore,
$\text{trace}(s\circ t)$ is a holomorphic function on $U_0$. The pairing
$$
\text{End}_P(E)\otimes \left(\text{End}^0_P(E)\otimes {\mathcal O}_{{\mathbb C}
{\mathbb P}^1}(D)\right)\, \longrightarrow \, {\mathcal O}_{
{\mathbb C}{\mathbb P}^1}
$$
defined by $s\otimes t \, \longmapsto\, \text{trace}(s\circ t)$
is nondegenerate. Hence we get an isomorphism
\begin{equation}\label{e9}
\text{End}^0_P(E)\otimes {\mathcal O}_{{\mathbb C}
{\mathbb P}^1}(D)\, \stackrel{\sim}{\longrightarrow}\, \text{End}_P(E)^*\, .
\end{equation}

For any locally defined section $s$ of $\text{End}_P(E)$, note that
$s\circ\theta -\theta\circ s$ is a locally defined section of
$\text{End}^0_P(E)\otimes K_{{\mathbb C}{\mathbb P}^1}\otimes
{\mathcal O}_{{\mathbb C}{\mathbb P}^1}(D)$.
Consider the two-term complex
$$
{\mathcal C}^{\bullet}\, :\, {\mathcal C}^0\, :=\, \text{End}_P(E)\,
\stackrel{[\cdot ,\theta]}{\longrightarrow}\, {\mathcal C}^1\, :=
\,\text{End}^0_P(E)\otimes K_{{\mathbb C}{\mathbb P}^1}\otimes
{\mathcal O}_{{\mathbb C}{\mathbb P}^1}(D)\, .
$$
The tangent space of $M^H_P$ at ${\bf E}$, defined in \eqref{epb} has the
following description in terms of hypercohomology:
\begin{equation}\label{e10}
T_{{\bf E}}M^H_P\, =\, {\mathbb H}^1({\mathcal C}^{\bullet})
\end{equation}
(see \cite[Section 6]{BR}).

Consider the short exact sequence of complexes
$$
\begin{matrix}
0&& 0\\
\Big\downarrow && \Big\downarrow\\
0 & \longrightarrow & \text{End}^0_P(E)\otimes K_{{\mathbb C}{\mathbb P}^1}
\otimes{\mathcal O}_{{\mathbb C}{\mathbb P}^1}(D)\\
\Big\downarrow && \Big\downarrow\\
\text{End}_P(E) & \stackrel{[\cdot ,\theta]}{\longrightarrow} &
\,\text{End}^0_P(E)\otimes K_{{\mathbb C}{\mathbb P}^1}\otimes
{\mathcal O}_{{\mathbb C}{\mathbb P}^1}(D)\\
\Big\downarrow && \Big\downarrow\\
\text{End}_P(E) &\longrightarrow & 0\\
\Big\downarrow && \Big\downarrow\\
0&& 0
\end{matrix}
$$
over ${\mathbb C}{\mathbb P}^1$.
It produces the following long exact sequence of hypercohomologies:
\begin{equation}\label{e11}
H^0({\mathbb C}{\mathbb P}^1,\, \text{End}^0_P(E)\otimes K_{{\mathbb C}
{\mathbb P}^1}\otimes{\mathcal O}_{{\mathbb C}{\mathbb P}^1}(D))
\,\longrightarrow\, {\mathbb H}^1({\mathcal C}^{\bullet})
\,\stackrel{\eta}{\longrightarrow}\,H^1({\mathbb C}{\mathbb P}^1,
\, \text{End}_P(E))\, .
\end{equation}

Using \eqref{e9} and Serre duality,
\begin{equation}\label{e12}
H^1({\mathbb C}{\mathbb P}^1, \, \text{End}_P(E))^*\, \cong \,
H^0\left({\mathbb C}{\mathbb P}^1,\, \text{End}^0_P(E)\otimes K_{{\mathbb C}
{\mathbb P}^1}\otimes{\mathcal O}_{{\mathbb C}{\mathbb P}^1}(D)\right)\, .
\end{equation}
Now consider the composition
$$
T_{{\bf E}}M^H_P\, =\, {\mathbb H}^1({\mathcal C}^{\bullet})
\,\stackrel{\eta}{\longrightarrow}\, H^1({\mathbb C}{\mathbb P}^1, \,
\text{End}_P(E))\,\stackrel{\theta\cdot}{\longrightarrow}\,
{\mathbb C}\, ,
$$
where $\theta\cdot v\, =\, \theta(v)$ (see \eqref{e12} for the
duality pairing), and the homomorphism $\eta$ is constructed in \eqref{e11}.
This composition defines an algebraic $1$-form
\begin{equation}\label{e14}
\lambda \, \in\, H^0(M^H_P, \, \Omega^1_{M^H_P})\, .
\end{equation}
The de Rham differential $d\lambda$ is the natural symplectic form
on $M^H_P$, which we refer to as the Higgs symplectic form  \cite{BR}.

It should be mentioned that $H^1({\mathbb C}{\mathbb P}^1, \,
\text{End}_P(E))$ parametrizes the infinitesimal deformations of
the parabolic vector bundle $(E\, ,\{\ell_i\}_{i=1}^n)$.
The homomorphism $\eta$ in \eqref{e11} is the forgetful
homomorphism that sends infinitesimal deformations of
the parabolic Higgs bundle $(E\, ,\{\ell_i\}_{i=1}^n, \theta)$
to the corresponding infinitesimal deformations of $(E\, ,
\{\ell_i\}_{i=1}^n)$ obtained by forgetting the Higgs field.

\subsection{The pullback of $\lambda$}

The total space $(({\mathbb C}^2)^{\oplus n})^*\times U$ of the cotangent
bundle of $U$ is equipped with the Liouville symplectic form. Let $\omega_0$
be the restriction to ${\mathcal U}_S$ (see \eqref{cus}) of this
Liouville symplectic form. Recall that $d\lambda$, where $\lambda$ is
constructed in \eqref{e14}, is the canonical holomorphic symplectic form on $M^H_P$.

\begin{theorem}\label{thm1}
Consider $\varphi$ constructed in \eqref{e6}. The pulled back form $\varphi^* d
\lambda\,=\, d\varphi^*\lambda$ on ${\mathcal U}_S$ coincides with $\omega_0$.
\end{theorem}

\begin{proof}
Take any point
\begin{equation}\label{e15}
\underline{y}\, =\, \left((y_1\, ,\cdots\, ,y_n)\, , (z_1\, ,
\cdots\, , z_n)\right)\, \in\, {\mathcal U}_S\,\subset\, (({\mathbb
C}^2)^{\oplus n})^* \times ({\mathbb C}^2)^{\oplus n}\, =\,
({\mathbb C}^2)^{\oplus n}\times ({\mathbb C}^2)^{\oplus n}\, ,
\end{equation}
where $y_i\, \in\, ({\mathbb C}^2)^*$, $z_i\, \in\, {\mathbb C}^2$. Let
\begin{equation}\label{z}
\varphi(\underline{y})
 = (E\, ,\{\ell_i\}_{i=1}^n\, , \theta)
\end{equation}
be the parabolic Higgs bundle.

Let $d\varphi\, :\, T{\mathcal U}_S\,\longrightarrow\,
TM^H_P$ be the differential of the morphism $\varphi$, where
$T {\mathcal U}_S$ and $TM^H_P$ are holomorphic tangent bundles.
We will compute the composition
\begin{equation}\label{e16}
T_{\underline{y}}{\mathcal U}_S \,\stackrel{d\varphi}{\longrightarrow}
\, T_{\varphi(\underline{y})}M^H_P\, \cong \,{\mathbb H}^1({\mathcal C}^{\bullet})
\,\stackrel{\eta}{\longrightarrow}\,H^1({\mathbb C}{\mathbb P}^1,
\, \text{End}_P(E))
\end{equation}
(see \eqref{e11} and \eqref{e10} for $\eta$ and the isomorphism
respectively). Take any
\begin{equation}\label{e17}
\underline{v}\, =\, \left((u_1\, ,\cdots\, ,u_n)\, , (v_1\, ,
\cdots\, , v_n)\right)\, \in\, T_{\underline{y}} {\mathcal U}_S\, ,
\end{equation}
where $u_i\, \in\, ({\mathbb C}^2)^*$ and
$v_i\, \in\, {\mathbb C}^2$. We want to describe 
the image of $\underline{v}$ under the composition in \eqref{e16}.

We recall from Section \ref{sec2} that $E$ is the trivial
vector bundle ${\mathbb C}^2\times {\mathbb C}{\mathbb P}^1\,
\longrightarrow\, {\mathbb C}{\mathbb P}^1$. For each $i\,\in\,
[1\, ,n]$, let
$$
\text{End}_{\mathbb C}({\mathbb C}^2) \, \longrightarrow\,
\text{Hom}_{\mathbb C}({\mathbb C}\cdot z_i\, ,
{\mathbb C}^2/({\mathbb C}\cdot z_i))
$$
be the surjective homomorphism obtained by restricting endomorphisms of
${\mathbb C}^2$ to the line ${\mathbb C}\cdot z_i$ generated by $z_i$, and
then projecting the image of this line to ${\mathbb C}^2/({\mathbb C}\cdot
z_i)$. We will consider $\text{Hom}_{\mathbb C}({\mathbb C}\cdot z_i\, ,
{\mathbb C}^2/({\mathbb C}\cdot z_i))$ as the torsion sheaf on
${\mathbb C}{\mathbb P}^1$ supported at the point $x_i\,\in\, D$.
The kernel of the homomorphism of coherent sheaves
$$
E\, =\,
{\mathbb C}^2\times {\mathbb C}{\mathbb P}^1\, \longrightarrow\,
\bigoplus_{i=1}^n \text{Hom}_{\mathbb C}({\mathbb C}\cdot z_i\, ,
{\mathbb C}^2/({\mathbb C}\cdot z_i))
$$
coincides with $\text{End}_P(E)$, because the quasiparabolic line
$\ell_i\,\subset\, E_{x_i}\, =\, {\mathbb C}^2$ at $x_i$
coincides with ${\mathbb C}\cdot z_i$. In other words, we get
a short exact sequence of coherent sheaves
$$
0\,\longrightarrow\, \text{End}_P(E)\,\longrightarrow\, \text{End}(E)\,
\longrightarrow\, \bigoplus_{i=1}^n \text{Hom}_{\mathbb C}({\mathbb C}
\cdot z_i\, ,{\mathbb C}^2/({\mathbb C}\cdot z_i))\,\longrightarrow\, 0
$$
over ${\mathbb C}{\mathbb P}^1$. Let
\begin{equation}\label{e18}
\bigoplus_{i=1}^n \text{Hom}_{\mathbb C}({\mathbb C} \cdot z_i\, ,
{\mathbb C}^2/({\mathbb C}\cdot z_i))\,\stackrel{\xi}{\longrightarrow}
\, H^1({\mathbb C}{\mathbb P}^1,\, \text{End}_P(E))
\end{equation}
be the map obtained by the long exact sequence of cohomologies associated to this
short exact sequence.

Consider the homomorphism ${\mathbb C}
\cdot z_i\, \longrightarrow\, {\mathbb C}^2$ defined by
$w \cdot z_i\, \longmapsto\, w \cdot v_i$ (see
\eqref{e17} for $v_i$). Its composition with the the
natural projection of ${\mathbb C}^2\, \longrightarrow\, {\mathbb C}^2/
({\mathbb C}\cdot z_i)$ yields a homomorphism
$$
\overline{v}_i\, :\, {\mathbb C} \cdot z_i\,
\longrightarrow\, {\mathbb C}^2/({\mathbb C}\cdot z_i)\, .
$$

Now it is straightforward to check that
$$
\left( \eta\circ d\varphi\right) (\underline{v})\, =\,\sum_{i=1}^n \xi
(\overline{v}_i)\, \in\, H^1({\mathbb C}{\mathbb P}^1,\, \text{End}_P(E))
$$
(see \eqref{e16} and \eqref{e18} for $\eta$ and $\xi$
respectively), where $\overline{v}_i$ is constructed above.

Consider the Higgs field $\theta$ in \eqref{z}. Using the
isomorphism in \eqref{e12}, it defines a functional on
$H^1({\mathbb C}{\mathbb P}^1, \, \text{End}_P(E))$. We want
to calculate $\theta(\xi (\overline{v}_i))\,\in\,\mathbb C$.

Now we will recall a property of the Serre duality pairing.
Let $V$ be an algebraic vector bundle over an irreducible
smooth complex projective curve $X$. Fix a point $x\,\in\, X$.
Let $S$ be a subspace of the vector space
$(V\otimes {\mathcal O}_X(x))_x$. Let $\widetilde{V}$ be the vector
bundle on $X$ that fits in the short exact sequence
$$
0\,\longrightarrow\, \widetilde{V}\,\longrightarrow\, V\otimes
{\mathcal O}_X(x)\,\longrightarrow\, (V\otimes {\mathcal O}_X(x))_x/S
\,\longrightarrow\, 0\, .
$$
Therefore, we have a short exact sequence of coherent sheaves
\begin{equation}\label{ses}
0\,\longrightarrow\, V\,\longrightarrow\, \widetilde{V}
\,\longrightarrow\, S\,\longrightarrow\, 0\, .
\end{equation}
Let
$$
S \, \stackrel{\beta}{\longrightarrow}\, H^1(X, \, V)
$$
be the homomorphism in the long exact sequence of cohomologies
associated to the short exact sequence in \eqref{ses}. Then for any
$$
\gamma\,\in\, H^0(X,\, V^*\otimes K_X)\, ,
$$
and any $w\,\in\, S$, the Serre duality pairing $\gamma(\beta(w))\,\in\,
\mathbb C$ coincides with $\gamma(x)(w)$; note that since $\gamma(x)\,\in
\, (V^*\otimes K_X)_x\, =\, ((V\otimes {\mathcal O}_X(x))_x)^*$
(Poincar\'e adjunction formula), we can evaluate $\gamma(x)$ on $w$.

{}From the above property of the Serre duality pairing it
follows immediately that
$$
\theta(\xi (\overline{v}_i))\, =\, y_i(v_i)\, ;
$$
recall from \eqref{e15} and \eqref{e17} that $y_i\, \in\, ({\mathbb C}^2)^*$
and $v_i\,\in\, {\mathbb C}^2$ respectively.

Consequently, the form $\varphi^*\lambda$ on ${\mathcal U}_S$
coincides with the following $1$-form $\lambda'$
on ${\mathcal U}_S$: for any point
$\underline{y}$ as in \eqref{e15} and any tangent
vector $\underline{v}$ at $\underline{y}$ as in \eqref{e17},
$$
\lambda'(\underline{v}) \, :=\, \sum_{i=1}^n y_i(v_i)\, .
$$

It is straightforward to check that $\lambda'(\underline{v})$
is the evaluation at $\underline{v}$ of the tautological one-form
on the total space of the holomorphic cotangent bundle
$(({\mathbb C}^2)^{\oplus n})^*\times U\, \longrightarrow\, U$.
Therefore, $d\lambda'$ coincides with the restriction to ${\mathcal U}_S$
of the Liouville symplectic form on $(({\mathbb C}^2)^{\oplus n})^*\times U$.
Consequently, $d\varphi^*\lambda$ coincides with $\omega_0$.
\end{proof}

The holomorphic symplectic form $\omega_0$ on $\mathcal{U}_S$ is easily seen to be $G$-invariant (see \eqref{eq:action}). This defines a natural holomorphic symplectic form $\omega$
on the hyperpolygon space $X(\alpha)=\mathcal{U}_S/\!\! / G$ and we obtain:

\begin{corollary}
The isomorphism $\overline{\varphi}\,:\,(X(\alpha),\omega)\,\longrightarrow\,
(M^H_P,d\lambda)$ from Proposition~\ref{prop:GM} is a symplectomorphism.
\end{corollary}

\medskip
\noindent
\textbf{Acknowledgements.}\, The first author wishes to thank
Instituto Superior T\'ecnico, where the work was carried out,
for its hospitality; his visit to IST was funded by the FCT project
PTDC/MAT/099275/2008.  This work was partially supported by the FCT projects PTDC/MAT/108921/2008,
PTDC/MAT/120411/2010 and the FCT grant SFRH/BPD/44041/2008.



\begin{thebibliography}{ZZZ}

\bibitem[BR]{BR} I. Biswas and S. Ramanan, \emph{An infinitesimal study of the
moduli of Hitchin pairs}, Jour. London Math. Soc. \textbf{49}
(1994), 219--231.

\bibitem[GM]{GM} L. Godinho and A. Mandini, \emph{Hyperpolygon spaces
and moduli spaces of parabolic Higgs bundles}, Adv. Math. (To appear),
http://arxiv.org/abs/1101.3241.

\bibitem[GH]{GH} P. B. Griffiths and J. Harris, \textit{Principles of 
algebraic geometry}, John Wiley \& Sons, New York, 1978.

\bibitem[Hi]{Hi} N. Hitchin, \emph{Stable bundles and integrable systems},
Duke Math. Jour. \textbf{54} (1987), 91--114.

\bibitem[Ko]{K} H. Konno, \emph{On the cohomology ring of the Hyper-K\"{a}hler
analogue of the polygon spaces}, Integrable systems, topology and physics (Tokyo, 
2000), 129--149, Contemp. Math. {\bf 309}, Amer. Math. Soc., Providence, RI, 2002. 

\bibitem[Ma]{Ma} M. Maruyama, \emph{Openness of a family of torsion free
sheaves}, Jour. Math. Kyoto Univ. \textbf{16} (1976),
627--637.

\bibitem[N]{N}
H. Nakajima, \emph{Quiver varieties and {K}ac-{M}oody algebras},
Duke Math. J., \textbf{91}, (1998), 515--560.
  
\end{thebibliography}
\end{document}